\documentclass[10pt, a4paper, reqno]{amsart}
\pdfoutput=1
\usepackage{amsmath,amsfonts,amssymb,amsthm,color}  
\usepackage{mathrsfs}

\usepackage[utf8]{inputenc}
\usepackage{graphicx}

\usepackage{marginnote}

\theoremstyle{plain}
\newtheorem{theorem}{Theorem}[section]

\newtheorem{proposition}[theorem]{Proposition}

\theoremstyle{definition}

\newtheorem{remark}[theorem]{Remark}

\numberwithin{equation}{section}
\newtheorem*{theorem*}{Theorem} 

\newcommand{\Z}{{\mathbb Z}}
\newcommand{\R}{{\mathbb R}}
\newcommand{\N}{{\mathbb N}}

\def\XXint#1#2#3{{\setbox0=\hbox{$#1{#2#3}{\int}$}
\vcenter{\hbox{$#2#3$}}\kern-.5\wd0}}

\DeclareMathOperator{\supp}{supp}

\providecommand{\norm}[1]{ \lVert#1  \rVert}

\title[Fractional maximal functions and Fourier multipliers]{Regularity of fractional maximal functions through Fourier multipliers}

\author{David Beltran, Jo\~ao Pedro Ramos and Olli Saari}
\date{\today}

\address{David Beltran, Basque Center for Applied Mathematics (BCAM), Bilbao, Spain}
	\email{dbeltran@bcamath.org}
	
\address{Jo\~ao Pedro Ramos, Mathematical Institute, 
	University of Bonn,
	Endenicher Allee 60, 53115, Bonn,
	Germany}
	\email{jpgramos@math.uni-bonn.de}
	
\address{Olli Saari, Mathematical Institute, 
	University of Bonn,
	Endenicher Allee 60, 53115, Bonn,
	Germany}
	\email{saari@math.uni-bonn.de}
	


\thanks{DB was supported by the ERCEA Advanced Grant 2014 669689 - HADE, the MINECO project MTM2014-53850-P, the Basque Government project IT-641-13, the Basque Government through the BERC 2014-2017 program, by Spanish Ministry of Economy and Competitiveness MINECO: BCAM Severo Ochoa excellence accreditation SEV-2013-0323. JPR acknowledges the Deutscher Akademischer Austauschdienst (DAAD) for funding. OS was supported by the Hausdorff Center for Mathematics and DFG grant CRC 1060.}

\begin{document}

\begin{abstract}
We prove endpoint bounds for derivatives of fractional maximal functions with either smooth convolution kernel or lacunary set of radii in dimensions $n \geq 2$. We also show that the spherical fractional maximal function maps $L^{p}$ into a first order Sobolev space in dimensions $n \geq 5$.   
\end{abstract}

\maketitle

\section{Introduction}

Define the fractional maximal function as 
\begin{equation*}
M_\alpha f (x) =  \sup_{t  > 0} \left \lvert \frac{t^{\alpha}}{|B(x,t)|} \int_{B(x,t)} f \, dy \right \rvert
\end{equation*}
for $\alpha \in [0,n)$. The study of its regularity properties was initiated in \cite{KS2003} by Kinnunen and Saksman. They proved the pointwise inequality
\begin{equation}\label{Kinnunen Saksman}
 |\nabla M_{\alpha}|f|(x)| \leq C M_{\alpha - 1}|f|(x), \quad \alpha \geq 1
\end{equation}
with a constant $C$ only depending on the dimension $n$ and $\alpha$. This inequality has two interesting consequences. First, $M_{\alpha}$ maps $L^{p}(\mathbb{R}^{n})$ into a first order Sobolev space. Second, as noted by Carneiro and Madrid \cite{CM2015}, the pointwise bound together with the Gagliardo--Nirenberg--Sobolev inequality implies
\begin{equation}
\label{intro:eq2}
\norm{\nabla M_{\alpha} f }_{L^{p}} \leq C \norm{ M_{\alpha-1} f }_{L^{p}} \leq C  \norm{  f }_{L^{n/(n-1)}} \leq C \norm{\nabla f}_{L^{1}}
\end{equation}
for $\alpha\geq 1$ and $p = n/(n-\alpha)$. When $\alpha \in (0,1)$, inequality \eqref{Kinnunen Saksman} no longer helps, and the conclusion of \eqref{intro:eq2} is an open problem. When $M_{\alpha}$ is replaced by its non-centred variant, the analogous result is due to Carneiro and Madrid \cite{CM2015} for $n = 1$ and Luiro and Madrid \cite{LM2017} for $f$ radial and $n \geq 2$. For other aspects of the regularity of fractional maximal functions, see e.g.~ \cite{HKKT2015,HKNT2013} and the references therein.

Our first result is a smooth variant of the inequality \eqref{intro:eq2} for $\alpha \in (0,1)$ and $n \geq 2$. Define the lacunary fractional maximal function as   
\[M_{\alpha}^{lac}f (x) := \sup_{ k \in \mathbb{Z} } \left \lvert \frac{2^{\alpha k}}{|B(0,2^{k})|} \int_{B(x,2^{k})} f \, dy \right \rvert. \]
For integrable $\varphi$ and $t > 0$, let $\varphi_t(x)=t^{-n}\varphi(x/t)$. Assume, for simplicity, that $\varphi$ is a positive Schwartz function and define the smooth fractional maximal function as 
\[M_{\alpha}^{\varphi}f(x) = \sup_{t > 0} t^{\alpha} | \varphi_{t} * f(x) | .\]
The smoothness requirement can be substantially relaxed, see $\S\S$\ref{sec:smooth}.
  
\begin{theorem}
\label{thm:main solid balls}
Let $f \in \dot{\mathrm{BV}}(\mathbb{R}^n)$ and suppose that $\alpha \in (0,1)$ and $n \geq 2$. Then for $\mathcal{M}_\alpha \in \{ M_\alpha^{lac}, M_\alpha^\varphi\}$, there exists a constant $C$ only depending on dimension $n$, $\alpha$ and $\varphi$ such that 
\[\norm{\nabla \mathcal{M}_\alpha f}_{L^{p}(\mathbb{R}^n)} \leq C | f |_{\mathrm{BV}(\mathbb{R}^n)}  \]
for $p = n/(n-\alpha)$. 
\end{theorem}

The proof of this theorem uses the $g$-function technique familiar from Stein's spherical maximal function theorem. The idea is to follow the scheme behind the short estimation \eqref{intro:eq2}. The Fourier transform is used to find a substitute for \eqref{Kinnunen Saksman} at the level of Besov spaces, from which the conclusion then follows by a refined Gagliardo--Nirenberg--Sobolev type embedding theorem \cite{CDDD2003}. The last step requires $n > 1$  whereas the smoothness condition on the maximal operator is imposed by Fourier analysis. We stress that the right hand side of the conclusion is BV norm instead of the considerably larger homogeneous Hardy--Sobolev norm one might first expect. The detailed proof is given in $\S$\ref{sec:solid balls}, and all necessary definitions can be found in $\S$\ref{sec:definitions}. To the best of our knowledge, Fourier transform techniques have not been exploited effectively in the study of endpoint regularity of maximal functions prior to this work.

The background of the question \eqref{intro:eq2} goes back to Kinnunen's theorem \cite{Kinnunen1997,KL1998} asserting that the Hardy--Littlewood maximal function is bounded in $W^{1,p}$ with $p > 1$. His result was later extended to $W^{1,1}$  in the form
\begin{equation}
\label{intro:w11}
\| \nabla M f \|_{L^1(\mathbb{R}^n)} \leq C \| \nabla f \|_{L^1(\mathbb{R}^n)}  
\end{equation}
by Tanaka \cite{Tanaka2002} when $n = 1$ and Luiro \cite{Luiro2017} when $n \geq 2$ and $f$ is radial. Here $M$ is the non-centred Hardy--Littlewood maximal function. The same inequality for $M_0$ (centred maximal function) was established by Kurka \cite{Kurka2010} when $n = 1$, and the question is open in dimensions $n \geq 2$. Kurka's theorem can be seen as the limiting case $\alpha = 0$ of \eqref{intro:eq2}.

In connection to \eqref{intro:w11}, maximal functions with smooth convolution kernels are better understood than the Hardy--Littlewood maximal function. Inequality \eqref{intro:w11} can be proved with sharp constant for many smooth kernels \cite{CFS2015,CS2013} whereas the best constant for centred Hardy--Littlewood maximal function is not known (for the non-centred maximal function \cite{AP2007} as well as for certain non-tangential maximal functions \cite{Ramos2017} the constant is one). Similarly, a Hardy--Sobolev bound corresponding to \eqref{intro:w11} is known for smooth maximal functions in all dimensions \cite{PPSS2017} whereas the progress for the standard maximal function is limited to the case of radial functions \cite{Luiro2017}. Finally, there are metric measure spaces where Kinnunen's theorem does not hold but suitable smoother maximal functions satisfy a Sobolev bound \cite{AK2010}. Theorem \ref{thm:main solid balls} can be seen as a part of this line of research attempting to understand \eqref{intro:eq2} and \eqref{intro:w11} first in the case of smooth maximal functions.

The second part of the paper studies the regularity of the spherical fractional maximal function
\begin{equation}
\label{sph:definition}
S_{\alpha} f(x) := \sup_{t > 0}  |t^{\alpha}  \sigma_t * f(x)| , 
\end{equation} 
where $\sigma_t$ is the normalized surface measure of the sphere $\partial B(0,t)$. For $\alpha = 0$, one recovers the spherical maximal function of Stein \cite{Stein1976} ($n \geq 3$) and Bourgain \cite{Bourgain1986} ($n = 2$). For $\alpha > 0$, $L^{p} \to L^{q}$ bounds for this operator follow from the work of Schlag \cite{Schlag1997} ($n=2$) and Schlag and Sogge \cite{SS1997} ($n \geq 3$). It is natural to ask if the fractional spherical maximal function has regularizing properties similar to \eqref{Kinnunen Saksman}. Our result in this direction is the following.

\begin{theorem}
\label{thm:spherical_technical}
Let $n \geq 5$, $n/(n-2) < p \leq q < \infty$ and 
\begin{equation*}
\alpha(p):=
\begin{cases}
\frac{n^2-2n-1}{n-1}  -\frac{2n}{p(n-1)} & \text{if} \quad  \frac{n}{n-2} < p \leq  \frac{n^2+1}{n^2-2n-1} \\
\frac{n-1}{p} & \text{if} \quad \frac{n^2+1}{n^2-2n-1} < p \leq n-1 .
\end{cases}
\end{equation*}
Assume that
\[\frac{1}{q}=\frac{1}{p}- \frac{\alpha-1}{n}, \qquad 1 \leq \alpha <  \alpha(p). \]
Then, for any $f \in L^p$, $S_{\alpha} f$ is weakly differentiable and 
\[\left \lVert \nabla S_{\alpha} f\right \rVert_{L^{q}} \lesssim \norm{f}_{L^{p}}.\]
\end{theorem}

The proof of this theorem is also based on the use of the Fourier transform. When $q \geq 2$, we study $L^{p} \to L^{q}$ estimates for a maximal multiplier operator in analogy with the estimates in \cite{Schlag1997,SS1997,Lee2003} for the spherical maximal function. Since Theorem \ref{thm:spherical_technical} is a statement at the derivative level, the corresponding multiplier enjoys worse Fourier decay than $\widehat{\sigma}$. This forces us to study the behavior in $L^{p}$ with large $p$ more carefully than what is needed to understand $L^{p}$ mapping properties of the spherical maximal function. We take advantage of the sharp local smoothing estimate for the wave equation in $L^{n-1}(\mathbb{R}^{n})$, which is available whenever $n \geq 5$ thanks to recent advances in decoupling theory (see \cite{BD2015,Garrigos2009,Garrigos2010,Laba2002,Wolff2000} and \cite{BHS2018,HNS2011,LS2013,MSS1992,Sogge1991} for more on decoupling and local smoothing estimates). We remark that results in $n=4$ could be obtained upon further progress on local smoothing estimates. 



\bigskip

\noindent \textbf{Acknowledgements.}
We would like to thank Juha Kinnunen for his question about regularising properties of the fractional spherical maximal function, which led to this work. We also thank Jonathan Hickman for discussions on the spherical maximal function and local smoothing estimates. 

\section{Notation and Preliminaries}\label{sec:definitions}

\subsection{Notation}
All function spaces are defined over $\mathbb{R}^{n}$, and it is written, for instance, $L^{2}$ for $L^{2}(\mathbb{R}^{n})$. The letter $C$ denotes a generic constant whose value may vary from line to line. Its dependency on other parameters will be clear from the context. The notation $A \lesssim B$ is used if $A \leq C B$ for such a constant $C$, and similarly $A \gtrsim B$ and $A \sim B$. The Fourier transform of a tempered distribution $f \in \mathcal{S}'$ is denoted by $\widehat{f}$ or $\mathcal{F}(f)$ and its inverse Fourier transform by $\mathcal{F}^{-1}(f)$ or $f^{\vee}$; in particular for a Schwartz function $f \in \mathcal{S}$,
\[\widehat{f}(\xi) = \mathcal{F}f (\xi) = \int_{\mathbb{R}^{n}} e^{ - 2 \pi i x \cdot \xi} f(x) \, dx .\]
Given any multi-index $\gamma \in \N^n$, $\partial^\gamma$ denotes
$$
\partial^\gamma f =\partial_{x_1}^{\gamma_1} \cdots \partial_{x_n}^{\gamma_n} f.
$$
For any $\alpha \in \R$, the notation $(-\Delta)^{\alpha/2}$ is taken to denote the operator associated to the Fourier multiplier $|\xi|^\alpha$.

\subsection{Besov spaces and Littlewood--Paley pieces}
\label{subsec:besov}
Given a smooth function $\psi \in C^\infty_c$ supported in $\{ \xi \in \R^n: 2^{-1} < |\xi | < 2\}$ and such that
$$
\sum_{j \in \mathbb{Z}} \psi(2^{-j} \xi) = 1
$$
for $\xi \neq 0$, let $f_j$ denote the Littlewood--Paley piece of $f$ at frequency $2^{j}$, given by $\widehat{f_j}=\widehat{f}\psi(2^{-j} \cdot)$. The Besov seminorm for $\dot{B}_{p,q}^{s}$ for $s \in \mathbb{R}$ and $p,q \in [1,\infty] $ is defined as 
\[\norm{f}_{\dot{B}_{p,q}^{s}} =  \Big( \sum_{j \in \mathbb{Z}} 2^{q j s} \norm{f_j}_{L^{p}}^{q} \Big)^{1/q};\] 
the seminorms defined through different Littlewood-Paley functions $\psi$ are comparable (see \cite[Chapter 6]{BL1976} for further details). 


\subsection{BV space}
A function $f$ is said to have bounded variation, and denoted by $f \in \dot{\mathrm{BV}}$, if its variation
$$
|f|_{\mathrm{BV}}:=\sup \Big\{ \int_{\R^n} f \: \mathrm{div} (g);  \:\: g \in C^1_c(\R^n, \R^n), \: \|  g\|_\infty \leq 1  \Big\}
$$
is finite, where $g=(g_1, \dots, g_n)$ and the $L^\infty$ norm is defined by
$$
\| g \|_{\infty}:= \| (\sum_{i=1}^n g_i^2 )^{1/2} \|_{L^\infty}.
$$
Note that if $f$ belongs to space $W^{1,1}$, integration by parts allows one to identify
$$
|f|_{\mathrm{BV}}=\int_{\R^n} |\nabla f|.
$$
See \cite[Chapter 5]{EG1992} for more.

\subsection{Finite differences}
\label{subsec:finitedifference}
Denote 
\[D^{h}f(x) = \frac{f(x+h)- f(x)}{|h|}.\]
Recall (see e.g~\cite[Chapter 5, $\S$5.8, Theorem 3]{Evans2010}) that if there is a finite constant $A$ such that  
\[ \left \lVert D^{h}f \right \rVert_{L^{p}} \leq A  \]
for all $h \in \mathbb{R}^{n}$, then the weak derivatives of $f$ exist and
\[\norm{\nabla  f }_{L^{p}} \leq C A \]
for a constant $C$ only depending on the dimension $n$. If $S$ is a sublinear operator that commutes with translations, then
\[ |D^{h}Sf| \leq |S D^{h}f|.\]
In particular, if $S$ is a maximal function and $f$ is a positive function, this allows us to reduce the question about differentiability to boundedness of a maximal multiplier for all Schwartz functions $f$. 

\section{Endpoint results} \label{sec:solid balls}

\subsection{A model result}\label{subsec:model}
It is instructive to start first with a model case for Theorem \ref{thm:main solid balls}. This consists in the study of the single scale version of the (rough) fractional maximal function $M_{\alpha}$, defined as
\[M_{\alpha}^{*}f = \sup_{1 \leq t \leq 2 } \left \lvert \frac{1}{|B(x,t)|} \int_{B(x,t)} f(y) \, dy \right  \rvert . \] 


\begin{theorem}
\label{thm:model_rough}
Let $0<\alpha < 1$, $p = n/(n-\alpha)$ and $n \geq 2$. Then there is a constant $C$ only depending on dimension $n$ and $\alpha$ such that for any $f \in \dot{B}_{p,1}^{1-\alpha}$
\[
\norm{ M_{\alpha}^{*} D^h f}_{L^{p}} \leq C \norm{f}_{\dot{B}_{p,1}^{1-\alpha}}
\]
uniformly on $h \in \mathbb{R}^{n}$.
\end{theorem}

By the discussion in $\S\S$\ref{subsec:finitedifference}, Theorem \ref{thm:model_rough} implies an $L^{p}$ bound for the gradient of $M_{\alpha}^{*}$. It will be shown in $\S\S$\ref{subsec:extensionFull} how the proof of the above estimate gives Theorem \ref{thm:main solid balls} for sightly smoother versions of the fractional maximal function, such as its lacunary version or maximal functions of convolution type with smooth kernels.

\begin{proof}
Write, for $f \in \mathcal{S}$,
\[M_{\alpha}^{*} (D^{h}f)(x) = \sup_{1 \leq t \leq 2} | \mathcal{F}^{-1}( (t|\xi|)^{\alpha} \widehat{1_{B(0,1)}}(t \xi)  \mathcal{F}( T_h (-\Delta)^{(1-\alpha)/2} f ) )| \]
where $T_h$ is the operator defined by 
\begin{equation}\label{Th definition}
\widehat{T_h g}(\xi) = \frac{e^{i \xi \cdot h} - 1}{|\xi||h|}\widehat{g}(\xi)=: a_h( \xi) \widehat{g}(\xi).
\end{equation}
Observe that $T_h$ is a bounded operator on $L^p$ uniformly in $h\in \mathbb{R}^{n}$ for all $1<p<\infty$ by the Mikhlin--Hörmander multiplier theorem (see, for instance \cite[Theorem 8.10]{Duo}); it is clear that
$$
|\partial^\gamma a_{h}(\xi)| \lesssim |\xi|^{-|\gamma|} \quad \quad \text{for all multi-indexes } \gamma \in \N_0^n
$$
with implicit constant independent of $h \in \R^n$. Thus, the operator $T_h$ plays no role in determining the range of boundedness for $M^*_\alpha D^h$.

Let $m(\xi) = |\xi|^{\alpha} \widehat{1_{B(0,1)}}( \xi)$ and $m_t(\xi):=m(t\xi)$ for all $t>0$. 
For each $j \in \Z$, let $f_j=\check{\psi}_j \ast f$ denote the Littlewood-Paley piece of $f$ around the frequency $2^{j}$ as in $\S\S$\ref{subsec:besov}. Assume momentarily that the following holds.

\begin{proposition}
\label{prop:Lp}
Let $g \in \mathcal{S}$. Then for $p = n/(n-\alpha)$ and $0 <\alpha < n/2$,
\[\norm{ \sup_{1 \leq t \leq 2} |\mathcal{F}^{-1}(m_t \widehat{g}_j ) |}_{L^{p}} \lesssim  (2^{j \alpha} 1_{\{j  \leq 0\}} + 1_{\{j > 0\}} ) \norm{g_j}_{L^{p}}. \]
\end{proposition}

Then the proof may be concluded as follows. Decomposing the function $f$ into frequency localised pieces $f_j$ and applying Proposition \ref{prop:Lp} to the function $g = T_h (-\Delta)^{(1-\alpha)/2} f$ one has
\begin{align}
\norm{\sup_{1\leq t \leq 2} | \mathcal{F}^{-1}(m_t \widehat{g} ) |  }_{L^{p}} &  \leq \sum_{j \in \Z}   \| \sup_{1\leq t \leq 2} |\mathcal{F}^{-1}(m_t \widehat{g}_{j}) | \|_{L^p} \notag \\
& \lesssim \sum_{j \in \Z} (2^{j \alpha } 1_{\{j \leq 0\}} + 1_{\{ j > 0 \}} )  \| g_{j} \|_{L^p} \notag \\
& \leq  \sum_{j \in \Z} 2^{j(1-\alpha)} \| f_{j} \|_{L^p}\sim \| f \|_{\dot{B}_{p,1}^{1-\alpha}}, \label{eq:Besov_bad}
\end{align}
where the last step follows from the $L^{p}$ boundedness of $T_h$ and Young's convolution inequality.

\begin{remark}
By Bernstein's inequality, $2^{j(1-\alpha)}\| f_j \|_{L^p} \lesssim 2^j \| f_j \|_{L^1}$, so one may further bound $\| f \|_{\dot{B}^{1-\alpha}_{p,1}} \lesssim \| f \|_{\dot{B}_{1,1}^1}$ in \eqref{eq:Besov_bad}.
\end{remark}

It remains to prove Proposition \ref{prop:Lp}. This is done by interpolating an $L^2$ bound with an $L^1-L^{1,\infty}$ bound as in the proof of the spherical maximal function theorem that can be found in the textbooks \cite[Chapter XI, $\S$3.3]{bigStein} or \cite[Chapter 5.5]{GrafakosClassical2014}. Writing 
\[
\mathcal{F}^{-1}(m_t \widehat{g}_j ) = t^{\alpha}  \mathcal{F}^{-1}(   \widehat{1_{B(0,1)}}(t \xi) (|\xi|^{\alpha}\widehat{g}_{j} ) ),
\]
it is clear that 
\[
\sup_{1 \leq t \leq 2} |\mathcal{F}^{-1}(m_t \widehat{g})| \lesssim  \sup_{1 \leq t \leq 2} |   t^{-n} 1_{B(0,t)} * ((-\Delta)^{\alpha/2} g) | \leq M((-\Delta)^{\alpha/2} g)
\]
where $M$ is the Hardy--Littlewood maximal function. Bounds on $M$ and Young's convolution inequality then imply
\begin{proposition}
Let $g \in \mathcal{S}$. Then
\begin{equation*}
\norm{ \sup_{1 \leq t \leq 2}| \mathcal{F}^{-1}(m_t \widehat{g}_j )| }_{L^{1,\infty}} \lesssim 2^{ j \alpha } \norm{g_j}_{L^{1}} .
\end{equation*}
\end{proposition}

The $L^{2}$ estimate follows by estimating the Fourier decay of $m$ after an application of a Sobolev embedding. This is the part of the proof that allows to take advantage of better symbols $m$ later in $\S\S$\ref{sec:smooth} so we write the proof in detail.

\begin{proposition}
\label{prop:L2}
Let $g \in \mathcal{S}$. Then
$$
\norm{ \sup_{1 \leq t \leq 2} |\mathcal{F}^{-1}( m_t \widehat{g}_j ) |}_{L^{2}} 
\lesssim  ( 2^{ j \alpha } 1_{\{j \leq 0\}}  + 2^{ j (- \frac{n}{2} + \alpha )} 1_{\{ j > 0\}} ) \norm{g_j}_{L^{2}}. 
$$
\end{proposition}
\begin{proof}
Let $\tilde{m}(\xi) = \xi \cdot \nabla m(\xi)$ and denote by $T_m$ and $T_{\tilde{m}}$ the operators associated to the multipliers $m$ and $\tilde{m}$. By the fundamental theorem of calculus,
\begin{align}
\sup_{1 \leq t \leq 2} |T_{m_t} g_j| 
	&\leq |T_{m} g_j| +  2 \left( \int_{1}^{2} |T_{ m_t} g_j | |T_{\tilde{m}_t} g_j | \frac{dt}{t} \right)^{1/2} \notag \\
	&\leq |T_{m} g_j|  + 2 \left(  \int_{1}^{2} |T_{ m_t} g_j |^{2} \frac{dt}{t} \right)^{1/4} \left(  \int_{1}^{2} |T_{\tilde{m}_t} g_j |^{2} \frac{dt}{t} \right)^{1/4} . \label{eq:sobolev embedding L2}
\end{align}
Taking $L^2$-norm in the above expression, an application of the Cauchy--Schwarz inequality and Fubini's theorem reduces the problem to compute the $L^\infty$ norm of $m\psi_j$ and $\tilde{m}\psi_j$.

Recall that $\widehat{1_{B(0,1)}}(\xi) = |2\pi \xi|^{-n/2} J_{n/2}(2 \pi |\xi|)$, where $J_{n/2}$ denotes the Bessel function of order $n/2$, and 
\[J_{n/2}(r) \lesssim  r^{n/2} 1_{\{r \leq 1\}} + r^{-1/2} 1_{\{r > 1\}};\]
see, for instance, \cite[Appendix B]{GrafakosClassical2014} for further details. This immediately yields
\begin{equation}
\label{eq:lacunary_single}
\norm{m \psi_j }_{L^{\infty}} \lesssim 2^{j\alpha} 1_{\{j \leq 0\}} +  2^{j( -\frac{n+1}{2} + \alpha )} 1_{\{j >0\}}.
\end{equation}
Concerning $\tilde{m}$, the relation
$$
\frac{d}{dr}[r^{-n/2} J_{n/2}(r)]=-r^{-n/2}J_{n/2 +1}(r)
$$
and a similar analysis to the one carried above leads to  
$$\norm{\tilde{m} \psi_j}_{L^{\infty}} \lesssim 2^{j\alpha} 1_{\{j \leq 0\}} + 2^{j( -\frac{n-1}{2} + \alpha )}1_{\{j >0\}}.$$
Putting both estimates together in \eqref{eq:sobolev embedding L2} concludes the proof.
\end{proof}

Proposition \ref{prop:Lp} now follows by interpolation, and the proof of the model case is complete.
\end{proof}

\subsection{Extension to the full supremum}\label{subsec:extensionFull}
%
From now on, we redefine $m$ to be Fourier transform of an integrable function smoother than $1_{B(0,1)}$. Momentarily assume $m$ satisfies
\begin{equation}\label{assumption extra decay}
\norm{ \sup_{1 \leq t \leq 2} | (m_t \widehat{g}_j )^\vee |}_{L^{p}} \lesssim  (2^{j \alpha} 1_{\{j  \leq 0\}} + 2^{-j\varepsilon} 1_{\{j > 0\}} ) \norm{g_j}_{L^{p}}
\end{equation}
with $p = n/(n-\alpha)$,
which we next show to be enough to conclude a bound as in Theorem \ref{thm:main solid balls}. The proof of \eqref{assumption extra decay} is postponed to $\S\S$\ref{sec:smooth}.

Inequality \eqref{assumption extra decay} rescales as
\begin{equation}\label{eq:rescaled}
\norm{ \sup_{2^{-k} \leq t \leq 2^{-k+1}} | (m_t \widehat{g}_{j+k} ) ^\vee|}_{L^{p}} \lesssim  (2^{j \alpha} 1_{\{j \leq 0\}} +  2^{-j\varepsilon}  1_{\{j  > 0\}}) \norm{g_{j+k}}_{L^{p}}. 
\end{equation}
In order to use this bound, break the full supremum over all possible scales and use the embedding $\ell^p \subseteq \ell^\infty$,
$$
\sup_{t >0} |  (m_t \widehat{g} )^\vee |  = \sup_{k \in \Z} \sup_{2^{-k} \leq t \leq 2^{-k+1}} | (m_t \widehat{g} )^\vee | \leq \Big( \sum_{k \in \Z}  \sup_{2^{-k} \leq t \leq 2^{-k+1}} | (m_t \widehat{g})^\vee |^p  \Big)^{1/p}.
$$
Taking $L^p$ norm and using \eqref{eq:rescaled}, we see
\begin{equation*}
\norm{\sup_{t >0} |  (m_t \widehat{g} )^\vee|  }_{L^{p}} \lesssim \sum_{j \in \Z} (2^{ j \alpha } 1_{\{j \leq 0\}} + 2^{- j \varepsilon} 1_{\{ j > 0 \}} ) \Big( \sum_{k \in \Z} \| g_{j+k} \|_{L^p}^p \Big)^{1/p} .
\end{equation*}
Using the geometric decay to sum in $j \in \Z$ and recalling
\[\norm{g_{j+k}}_{L^{p}} = \norm{ (-\Delta)^{(1-\alpha)/2} f_{j+k} }_{L^{p}} \lesssim  2^{(j+k)(1-\alpha)}\norm{f_{j+k}}_{L^{p}}, \]
we obtain
$$
\Big( \sum_{k \in \Z} \| g_{j+k} \|_{L^p}^p \Big)^{1/p} \lesssim \| f \|_{\dot{B}^{1-\alpha}_{p,p}}.
$$
We then claim
\begin{equation}\label{eq:Besov_BV}
\| f \|_{\dot{B}^{1-\alpha}_{p,p}} \lesssim | f |_{\mathrm{BV}}
\end{equation}
for $n>1$ and $0<\alpha < n/2$. This will follow from a Gagliardo--Nirenberg--Sobolev type inequality.
\begin{proposition}[\cite{CDDD2003}]
Assume $\gamma>1$ or $\gamma<1-1/n$, and let $(s,q)$ satisfy $(s-1)q'/n=\gamma-1$ for some $1 < q \leq \infty$, where $1/q+1/q'=1$. Then, for any $0< \theta < 1$,
$$
\| f \|_{\dot{B}^t_{p,p}} \lesssim  \| f \|_{\dot{B}_{q,q}^s}^{1-\theta} | f|_{\mathrm{BV}}^\theta
$$
where $\frac{1}{p}=\frac{1-\theta}{q} + \theta$ and $t=(1-\theta)s + \theta$.
\end{proposition}

Indeed, taking $\gamma=0$, $s=1-n/2$ and $\theta=1-2\alpha/n$, which are admissible for $n>1$ and $0<\alpha<n/2$, one has
$$
\| f \|_{\dot{B}_{p,p}^{1-\alpha}} \lesssim \| f \|_{\dot{B}^{1-n/2}_{2,2}}^{1-\theta} |f|_{\mathrm{BV}}^\theta.
$$
Applying Bernstein's and Minkowski's inequalities as well as Littlewood--Paley theory, we see
\begin{align*}
 \| f \|_{\dot{B}^{1-\frac{n}{2}}_{2,2}} & \sim \Big( \sum_{j \in \Z} 2^{2j(1-\frac{n}{2})} \| f_j \|_{L^{2}}^2 \Big)^{1/2} \lesssim  \Big( \sum_{j \in \Z} 2^{2j(1-\frac{n}{2})} 2^{2jn(\frac{n-1}{n}-\frac{1}{2})} \| f_j \|_{L^{\frac{n}{n-1}}}^2 \Big)^{1/2} \\
 &=  \Big( \sum_{j \in \Z} \| f_j \|_{L^{\frac{n}{n-1}}}^2 \Big)^{1/2} \leq \| \Big( \sum_{j \in \Z} |f_j|^2 \Big)^{1/2} \|_{L^{\frac{n}{n-1}}} \sim \| f \|_{L^{\frac{n}{n-1}}}.
\end{align*}
Inequality \eqref{eq:Besov_BV} then follows from the Gagliardo--Nirenberg--Sobolev inequality \cite[Theorem 5.6.1.~(i)]{EG1992}, and we conclude
$$
\| \sup_{1\leq t \leq 2} |\mathcal{F}^{-1}(m_t \widehat{g}) | \|_{L^p} \lesssim \| f \|_{L^{n/(n-1)}}^{1-\theta} |f|_{\mathrm{BV}}^\theta \lesssim |f|_{\mathrm{BV}}.
$$
Thus it suffices to verify \eqref{assumption extra decay}. This is done separately in the cases when $m$ comes from a smooth kernel and when the maximal function is lacunary.

\subsection{Smooth kernel}
\label{sec:smooth}
Define the smooth fractional maximal function as follows. Let $\epsilon > 0$. Let $\varphi$ be a positive function with radial $L^{1}$ majorant such that $\widehat{\varphi}(\xi) \lesssim_{\varphi}(1+ |\xi|)^{- n/2 - \epsilon}$. For instance, any positive Schwartz function or even 
\[\varphi(x) = (1-|x|^{2})_{+}^{\epsilon} \]
with $\epsilon > 0$ will do (see Appendix B.5 in \cite{GrafakosClassical2014}). The subscript denotes the positive part as $f_+ = f \cdot 1_{\{f > 0\}}$. Now we want to analyse $M^{\varphi}_{\alpha},$ as defined in the introduction. A repetition of the proof of Proposition \ref{prop:L2} gives the $L^{2}$ bound
\[\norm{ \sup_{1 \leq t \leq 2} |\mathcal{F}^{-1}( (t |\xi|)^{\alpha} \widehat{\varphi}(t\xi) \widehat{g}_j ) |}_{L^{2}} \lesssim  ( 1_{\{j \leq 0\}}2^{j \alpha}  +  1_{\{j > 0\}}2^{j(- \frac{n}{2} + \alpha - \epsilon )} )  \norm{g_j}_{L^{2}}.\]
The $\epsilon$-decay gain in the above estimate continues to hold on $L^{n/(n-\alpha)}$, so the extra decay assumption \eqref{assumption extra decay} is satisfied for smooth convolution kernels. By $\S\S$\ref{subsec:extensionFull}, Theorem \ref{thm:main solid balls} holds in this case.

\subsection{Lacunary set of radii}
Similarly, there is a gain in the $L^2$ estimate when we study the lacunary fractional maximal function. Now $m(\xi) = |\xi|^{\alpha} \widehat{1_{B(0,1)}}( \xi)$ and
\[c_n M_{\alpha}^{lac}f (x)=  \sup_{k \in \Z} | 2^{k\alpha-nk} \int_{B(x,2^{k})} f(y) \, dy| \leq \Big( \sum_{k \in \Z} | 2^{k\alpha-nk} \int_{B(x,2^{k})} f(y) \, dy|^p \Big)^{1/p} \]
so that it suffices to use a bound for a single dilate instead of a supremum bound. Thus, it is enough to use  \eqref{eq:lacunary_single} to replace Proposition \eqref{prop:L2} by
\[
\| |\mathcal{F}^{-1} (m \widehat{g}_j)| \|_{L^2} \lesssim  ( 2^{j \alpha} 1_{\{j \leq 0\}}  +  2^{j ( - \frac{n+1}{2} +  \alpha ) } 1_{\{j > 0\}}  )  \norm{g_j}_{L^{2}},
\]
which has an extra $1/2$-decay compared to Proposition \ref{prop:L2}. 
After interpolation, this leads to an $\varepsilon$-decay gain in the $L^{n/(n-\alpha)}$ estimate so that \eqref{assumption extra decay} (without supremum) and Theorem \ref{thm:main solid balls} for lacunary set of radii follow.

\section{Proof of Theorem \ref{thm:spherical_technical}}
Recall the definition \eqref{sph:definition}.
By the characterisation through finite differences described in $\S$2, the sublinearity of $S_\alpha$ and by density, it suffices to prove 
$$
\|S_{\alpha}D^{h} f \|_{L^{q}} \lesssim \| f \|_{L^{p}}
$$ 
for all Schwartz functions $f$ uniformly in $h \in \R^n$.

Observe that by means of Fourier transform,
\begin{equation}\label{eq:in Fourier side}
S_\alpha D^h f (x)= \sup_{t > 0} \left \lvert \mathcal{F}^{-1} \left( t^{\alpha} |\xi| \widehat{\sigma}(t \xi)  \mathcal{F} (T_{h}f) (x) \right) \right \rvert,
\end{equation}
where $T_h$ is the Fourier multiplier operator \eqref{Th definition}. As described in $\S\S$\ref{subsec:model}, $T_h$ is bounded on $L^p$ for all $1<p<\infty$ uniformly in $h \in \R^n$ by the Mikhlin--H\"ormander multiplier theorem, so it plays no role in determining the boundedness range for $S_\alpha D^h$; for this reason, $T_h f$ is identified with $f$ in the rest of this section.

\subsection{The case $q \geq 2$}
\label{sec:q>2}
It is enough to consider the single scale version of the maximal function in \eqref{eq:in Fourier side}: suppose we can prove
\begin{equation}
\label{sph:ss_red}
\norm{ \sup_{1 \leq t \leq 2} \lvert \mathcal{F}^{-1} ( t^{\alpha} |\xi| \widehat{\sigma}(t \xi) \widehat{f_j} )  \rvert}_{L^q} \lesssim  (2^{j s_1 } 1_{\{ j \leq 0 \}} + 2^{-j s_2} 1_{\{ j > 0\}} ) \| f_j \|_{L^p}
\end{equation}
for $s_1, s_2 >0$. Then rescaling gives
$$
\norm{ \sup_{2^{-k} \leq t \leq 2^{-k+1}} \lvert \mathcal{F}^{-1} ( t^{\alpha} |\xi| \widehat{\sigma}(t \xi) \widehat{f_{j+k}} )  \rvert}_{L^q} \lesssim  (2^{j s_1 } 1_{\{ j \leq 0 \}} + 2^{-j s_2} 1_{\{ j > 0\}} ) \| f_{j+k} \|_{L^p}
$$
under the relation $\frac{1}{q}=\frac{1}{p}-\frac{\alpha-1}{n}$, and arguing as in $\S\S$\ref{subsec:extensionFull}
\begin{multline*}
\norm{\sup_{t > 0} |\mathcal{F}^{-1}(t^\alpha |\xi| \widehat{\sigma}(t\xi) \widehat{f}) | }_{L^{q}} \lesssim  \sum_{j \in \mathbb{Z}} (2^{j s_1 } 1_{\{ j \leq 0 \}} + 2^{-j s_2} 1_{\{ j > 0\}} ) \Big( \sum_{k \in \mathbb{Z}}  \norm{f_{j+k}}_{L^{p}} ^{q} \Big)^{1/q} \\
 \lesssim \| f \|_{L^p}
\end{multline*}
where the last inequality follows from Minkowski's inequality $(q \geq p $); controlling $\ell^{q}$ norm by $\ell^{2}$ norm, and applying Littlewood--Paley theory to see the inner sum as $L^{p}$ norm of $f$. The sum in $j$ converges as $s_1, s_2 >0$. Hence it suffices to prove \eqref{sph:ss_red}.

For low frequencies $j \leq 0$, we can use domination by the Hardy--Littlewood maximal function, Young's convolution inequality and Bernstein's inequality to see
\[\norm{ \sup_{1 \leq t \leq 2} \lvert \mathcal{F}^{-1} ( t^{\alpha} |\xi| \widehat{\sigma}(t \xi) \widehat{f_j} )  \rvert}_{L^q} \lesssim \norm{  M(-\Delta)^{1/2} f_j}_{L^q} \lesssim 2^{j(1+\alpha)} \norm{f_j}_{L^p}.  \]
Hence it suffices to prove \eqref{sph:ss_red} for $j > 0$.

\subsection{A local smoothing estimate}
The Fourier transform of the spherical measure is
\[
\widehat{\sigma}(\xi)=2\pi |\xi|^{-\frac{n-2}{2}} J_{\frac{n-2}{2}}(2\pi |\xi|)=\sum_{\pm} a_{\pm} (\xi) e^{\pm 2 \pi i|\xi|},
\]
where the symbols $a_{\pm}$ are in the class $S^{-(n-1)/2}$, that is
$$
|\partial_\xi^\gamma a_{\pm} (\xi)| \lesssim (1+|\xi|)^{-\frac{n-1}{2} - |\gamma|}
$$
for all multi-indices $\gamma \in \N^n_0$ (c.f. \cite[Chapter VIII]{bigStein}). Hence
\[
\mathcal{F}^{-1}( \widehat{\sigma}(t\xi) \widehat{f}) = \sum_{\pm} \int_{\mathbb{R}^{n} } e^{2\pi i ( x \cdot \xi  \pm  t|\xi|)}  a_\pm( t \xi)  \widehat{f}(\xi) \, d\xi,
\]
so that the connection to half-wave propagator $ e^{it \sqrt{-\Delta}} f (x):=\int_{\R^n} e^{i x \cdot \xi} e^{ i t |\xi| } \widehat{f}(\xi) d\xi$ is evident. We will quote the following result:
\begin{proposition}[Consequence of \cite{BD2015}]
\label{thm:local smoothing so far} 
For $n \ge 2, \, s \in \R,$ 
\begin{equation*}
\Big(\int_{1}^2 \| e^{it \sqrt{-\Delta}} f  \|_{L^{p}_{s-s_p + \theta}(\R^n)}^p  dt \Big)^{1/p} \lesssim \| f \|_{L^p_s(\mathbb{R}^{n})}
\end{equation*}
holds for $0 \le \theta < \frac{1}{p}$ and $s_p = (n-1)\big(\frac{1}{2} - \frac{1}{p}\big)$ whenever $p \geq \frac{2(n+1)}{n-1}$.
\end{proposition}

This can be found as Corollary 1.3 (i) in \cite{Garrigos2009} knowing that the conjectured value of $p_d$ in Table 1 of that paper has later been verified by \cite{BD2015}.

\begin{proposition}
\label{sph:prop}
Let $g$ be a Schwartz function and $j > 0$. For any $\epsilon > 0$
\[  \norm{ \sup_{1 \le t \le 2}  | \sigma_t * g_j | }_{L^{n-1}} \lesssim_{\epsilon} 2^{j ( \epsilon - 1)}  \norm{g_j}_{L^{n-1}}  .\]
\end{proposition}
\begin{proof}
For $j>0$ and a smooth bump $\chi$ around $[1,2]$, we have
\begin{align*} 
\| \sup_{1 \le t \le 2}  | \sigma_t * g_j | \|_{L^{n-1}(\mathbb{R}^{n})} & \lesssim \| (1+\sqrt{-\partial_t^2})^r  \chi \cdot \sigma_t * g_j  \|_{L^{n-1}(\R^{n+1})} \cr 
											    & \lesssim 2^{j\left(r + s_p - \theta -\frac{n-1}{2} + \epsilon \right)} \| g_j  \|_{L^{n-1}(\R^n)}  \cr 
\end{align*}
where we used Sobolev embedding with $r > 1/(n-1)$, Proposition \ref{thm:local smoothing so far} with $p=n-1$ as well as Young's convolution inequality. Simplifying the exponent in accordance with  Proposition \ref{thm:local smoothing so far} 
\footnote{The full strength of \cite{BD2015} is only needed when $n=5$. When $n \geq 6$, the earlier results from \cite{Laba2002} will already do.}, 
we obtain the claim.
\end{proof}

\subsection{$L^{p} \to L^{q}$ estimates}
To finish the proof of \eqref{sph:ss_red}, we prove $L^{p} \to L^{q}$ estimates following  the interpolation scheme of Lee \cite{Lee2003} enhanced with the sharp local smoothing estimate. Denote 
\[S_j^{*} f (x) := \sup_{1 \le t \le 2} |\mathcal{F}^{-1}(   \widehat{\sigma}( t \xi ) |\xi| \widehat{f_j}(\xi))(x)|, \]
where $\widehat{f_j} = \widehat{f} \psi_j$ still stands for Fourier localization at the level of a Littlewood--Paley piece of frequency $2^j$. 

\begin{proposition}
Let $P$ be the open convex polygon with vertices
\begin{align*}
A&= \Big( \frac{n-2}{n}, \frac{2}{n} \Big), \quad B=\Big(\frac{n^2-2n-1}{n^2+1}, \frac{2(n-1)}{n^2+1} \Big)\\
C&=\Big(\frac{1}{n-1}, \frac{1}{n-1} \Big), \quad D= \Big( \frac{n-2}{n}, \frac{n-2}{n} \Big).
\end{align*}
Then
\[\norm{ S_j^{*} f }_{L^q} \lesssim 2^{- \varepsilon j} \norm{ f_j }_{L^p}\]
for some $\varepsilon > 0$ and all $j > 0$ provided that $(1/p,1/q) \in P$.
\end{proposition}

\begin{proof}
Since $\supp  ( \widehat{\sigma}\cdot \psi_j ( t \cdot )) \subset \{ |\xi| \sim 2^{j} \}$, we can assume that $\widehat{f}$ is supported in an annulus around $|\xi| = 2^{j}$. We use the following bounds: 
\begin{align}
\norm{ S_j^{*}f }_{L^1} &\lesssim 2^{2j} \norm{ f }_{L^1} \nonumber \\
\norm{ S_j^{*}f }_{L^\infty} &\lesssim 2^{2j} \norm{ f }_{L^1} \nonumber \\
\label{lplq}
\norm{ S_j^{*}f }_{L^{n-1}} &\lesssim_{\delta} 2^{j \delta} \norm{ f }_{L^{n-1}}, \quad \textrm{for all} \ \delta > 0 \\
\norm{ S_j^{*}f }_{L^2} &\lesssim 2^{- \frac{n-4}{2}j} \norm{ f }_{L^2} \nonumber \\
\norm{ S_j^{*}f }_{L^{\frac{2(n+1)}{n-1}}} &\lesssim 2^{- j \frac{n^2-4n-3}{2n+2}} \norm{ f }_{L^2} \nonumber.
\end{align}  
To verify \eqref{lplq}, use Proposition \ref{sph:prop} as well as Young's convolution inequality to obtain
\[\norm{ S_j^{*}f }_{L^{n-1}} \lesssim_\delta 2^{-j(1- \delta)}  \norm{ (-\Delta)^{1/2} f }_{L^{n-1}} \lesssim 2^{j \delta} \norm{ f }_{L^{n-1}}.\]
The other inequalities follow similarly, that is, by borrowing the corresponding bounds for the spherical maximal function (inequalities (1.7)--(1.10) in \cite{Lee2003}), and applying Young's convolution inequality. Interpolating the bounds above, we obtain the claimed proposition. 
\end{proof}

For each $p > 1$, we want to find the values of $\alpha$ such that $(1/p,1/q) \in P$ when $(\alpha-1)/n = 1/p-1/q$ and $q \geq 2$. When $q \geq 2$ is assumed, this happens when
\[
\frac{n}{n-2} < p \leq  \frac{n^2+1}{n^2-2n-1}, \quad \alpha < \frac{n^2-2n-1}{n-1}  -\frac{2n}{p(n-1)}
\]
or 
\[
\frac{n^2+1}{n^2-2n-1} < p \leq n-1 ,\quad \alpha < \frac{n-1}{p} .
\]
This concludes the proof for the case $q \geq 2$.  Notice that the restriction $q \ge 2$ is not dictated by validity of $L^{p} \to L^{q}$ estimates but it was required in order to upgrade the single scale bounds to bounds for the full maximal operator in $\S\S$\ref{sec:q>2}.

\subsection{The case $q \leq 2$}

Next we remove the assumption $q \geq 2$. Let 
\[T^{*} f(x) = \sup_{t > 0} |\mathcal{F}^{-1} ( (t|\xi|)^{\alpha} \widehat{\sigma}(t\xi) \widehat{f}(\xi))(x)|.\] 
The operator $S_{\alpha}$ in \eqref{eq:in Fourier side} can be written
\[S_{\alpha} = T^* I_{\alpha-1}  T_h f \] 
where $\widehat{I_{\alpha-1}f} = |\xi|^{1-\alpha} \widehat{f}$ is the Riesz potential of order $\alpha -1$ and $T_h$ are as in \eqref{Th definition}. As discussed in $\S\S$\ref{subsec:model}, $T_h$ are bounded in $L^{p}$ for all $p > 1$. Also, by the Hardy--Littlewood--Sobolev inequality $I_{\alpha -1}$ is bounded $L^p \to L^q,$ for $p,q$ obeying $\frac{\alpha-1}{n} = \frac{1}{p}-\frac{1}{q}$. Therefore, it is enough to analyse the operator $T^*$. 

Let $m(\xi) = |\xi|^{\alpha} \widehat{\sigma}(\xi)$ and take a Littlewood--Paley function $\psi$ (as in $\S$\ref{sec:definitions}). We define $m_{1} = \sum_{j > 0} \psi_j m$ and $m_0 = \sum_{j \leq 0} \psi_j m$. Take $T_j^{*}$ to be as $T^*$ but $m$ replaced by $m_j$. Then
\[
 T^*f \le T_0^{*}f+ T^*_1f.
\]
We first bound $T_0^{*}$. A straightforward computation shows that $m_0$ is bounded and for any multi-index $\beta \in \mathbb{N}^{n}$ with $|\beta| = k$, $k \leq n +1$
\[ | \partial_{\xi}^{\beta} m_0(\xi) | \lesssim  |\xi|^{\alpha-k}   \]
so that 
\[\norm{ (1 + | \cdot |)^{n+1}  \mathcal{F}^{-1}( m_0 )  }_{L^{\infty}} \lesssim 1 .  \] 
Consequently 
\[T_0^{*}f \lesssim Mf\]
and boundedness in any $L^{p}$ with $p> 1$ follows from that of the Hardy--Littlewood maximal function.

To bound $T_1^{*}$, we use a part of Theorem B from \cite{RdF1986}:
\begin{theorem}[Rubio de Francia \cite{RdF1986}]
Let $m$ be a function in $C^{s+1}(\R^n)$ for some integer $s > n/2$ such that $|D^{\gamma}m(\xi)| \lesssim |\xi|^{-a},$ for all $|\gamma| \le s+1$. Suppose also that $a > \frac{1}{2}.$ Then the maximal 
multiplier operator $T^* f := \sup_{t > 0} |\mathcal{F}^{-1}( m(t \cdot) \widehat{f})|$ is bounded in $L^r,$ for 
\[ 
\frac{2n}{n+2a -1} < r \le 2.
\]
\end{theorem}

Since $\sum_{j > 0} \psi_j m $ is smooth and satisfies $|D^{\gamma}m(\xi)| \lesssim |\xi|^{-a},$ for all $\gamma \in \mathbb{N}^{n}$ with $a = \frac{n-1}{2} - \alpha$, we can apply the theorem to conclude the proof whenever 
\[ 
 \frac{2n}{2n - 2 -2\alpha} < q \le 2, \quad a > \frac{1}{2}
\]
which is equivalent to $p > \frac{n}{n-2}$ and $\alpha < \frac{n-2}{2} < \alpha(p)$. However, given $p > \frac{n}{n-2}$, the condition $\alpha < \frac{n-2}{2}$ is automatically satisfied whenever $q \leq 2$. Hence $\alpha < \alpha(p)$ is an active constraint only when $q > 2$.

\qed
 
\bibliography{Reference}

\begin{thebibliography}{10}

\bibitem{AK2010}
D.~Aalto and J.~Kinnunen.
\newblock The discrete maximal operator in metric spaces.
\newblock {\em J. Anal. Math.}, 111(1):369--390, 2010.

\bibitem{AP2007}
J.~M. Aldaz and J.~P\'erez~L\'azaro.
\newblock Functions of bounded variation, the derivative of the one dimensional
  maximal function, and applications to inequalities.
\newblock {\em Trans. Amer. Math. Soc.}, 359(5):2443--2461, 2007.

\bibitem{BHS2018}
D.~Beltran, J.~Hickman, and C.~D. Sogge.
\newblock Wolff-type inequalities and sharp local smoothing estimates for wave
  equations on manifolds.
\newblock {\verb+arxiv.org/abs/1801.06910+} (2018).

\bibitem{BL1976}
J.~Bergh and J.~L\"ofstr\"om.
\newblock {\em Interpolation spaces. {A}n introduction}.
\newblock Springer-Verlag, Berlin-New York, 1976.
\newblock Grundlehren der Mathematischen Wissenschaften, No. 223.

\bibitem{Bourgain1986}
J.~Bourgain.
\newblock Averages in the plane over convex curves and maximal operators.
\newblock {\em J. Analyse Math.}, 47:69--85, 1986.

\bibitem{BD2015}
J.~Bourgain and C.~Demeter.
\newblock The proof of the {$l^2$} decoupling conjecture.
\newblock {\em Ann. of Math. (2)}, 182(1):351--389, 2015.

\bibitem{CFS2015}
E.~Carneiro, R.~Finder, and M.~Sousa.
\newblock On the variation of maximal operators of convolution type ii.
\newblock To appear in \emph{Rev. Mat. Iber.},
  {\verb+arxiv.org/abs/1512.02715+} (2015).

\bibitem{CM2015}
E.~Carneiro and J.~Madrid.
\newblock Derivative bounds for fractional maximal functions.
\newblock {\em Trans. Amer. Math. Soc.}, 369(6):4063--4092, 2017.

\bibitem{CS2013}
E.~Carneiro and B.~F. Svaiter.
\newblock On the variation of maximal operators of convolution type.
\newblock {\em J. Funct. Anal.}, 265(5):837--865, 2013.

\bibitem{CDDD2003}
A.~Cohen, W.~Dahmen, I.~Daubechies, and R.~DeVore.
\newblock Harmonic analysis of the space {BV}.
\newblock {\em Rev. Mat. Iberoamericana}, 19(1):235--263, 2003.

\bibitem{Duo}
J.~Duoandikoetxea.
\newblock {\em Fourier analysis}, volume~29 of {\em Graduate Studies in
  Mathematics}.
\newblock American Mathematical Society, Providence, RI, 2001.
\newblock Translated and revised from the 1995 Spanish original by David
  Cruz-Uribe.

\bibitem{Evans2010}
L.~C. Evans.
\newblock {\em Partial differential equations}, volume~19 of {\em Graduate
  Studies in Mathematics}.
\newblock American Mathematical Society, Providence, RI, second edition, 2010.

\bibitem{EG1992}
L.~C. Evans and R.~F. Gariepy.
\newblock {\em Measure theory and fine properties of functions}.
\newblock Studies in Advanced Mathematics. CRC Press, Boca Raton, FL, 1992.

\bibitem{Garrigos2009}
G.~Garrig\'os and A.~Seeger.
\newblock On plate decompositions of cone multipliers.
\newblock {\em Proc. Edinb. Math. Soc. (2)}, 52(3):631--651, 2009.

\bibitem{Garrigos2010}
G.~Garrig\'os and A.~Seeger.
\newblock A mixed norm variant of {W}olff's inequality for paraboloids.
\newblock In {\em Harmonic analysis and partial differential equations}, volume
  505 of {\em Contemp. Math.}, pages 179--197. Amer. Math. Soc., Providence,
  RI, 2010.

\bibitem{GrafakosClassical2014}
L.~Grafakos.
\newblock {\em Classical {F}ourier analysis}, volume 249 of {\em Graduate Texts
  in Mathematics}.
\newblock Springer, New York, third edition, 2014.

\bibitem{HKKT2015}
T.~Heikkinen, J.~Kinnunen, J.~Korvenp{\"a}{\"a}, and H.~Tuominen.
\newblock Regularity of the local fractional maximal function.
\newblock {\em Arkiv f{\"o}r Matematik}, 53(1):127--154, 2015.

\bibitem{HKNT2013}
T.~Heikkinen, J.~Kinnunen, J.~Nuutinen, and H.~Tuominen.
\newblock Mapping properties of the discrete fractional maximal function in
  metric measure spaces.
\newblock {\em Kyoto J. Math.}, 53(3):693--712, 2013.

\bibitem{HNS2011}
Y.~Heo, F.~Nazarov, and A.~Seeger.
\newblock Radial {F}ourier multipliers in high dimensions.
\newblock {\em Acta Math.}, 206(1):55--92, 2011.

\bibitem{Kinnunen1997}
J.~Kinnunen.
\newblock The {H}ardy-{L}ittlewood maximal function of a {S}obolev function.
\newblock {\em Israel J. Math.}, 100:117--124, 1997.

\bibitem{KL1998}
J.~Kinnunen and P.~Lindqvist.
\newblock The derivative of the maximal function.
\newblock {\em J. reine angew. Math.}, 503:161--167, 1998.

\bibitem{KS2003}
J.~Kinnunen and E.~Saksman.
\newblock Regularity of the fractional maximal function.
\newblock {\em Bull. London Math. Soc.}, 35(4):529--535, 2003.

\bibitem{Kurka2010}
O.~Kurka.
\newblock On the variation of the {H}ardy-{L}ittlewood maximal function.
\newblock {\em Ann. Acad. Sci. Fenn. Math.}, 40(1):109--133, 2015.

\bibitem{Laba2002}
I.~{\L}aba and T.~Wolff.
\newblock A local smoothing estimate in higher dimensions.
\newblock {\em J. Anal. Math.}, 88:149--171, 2002.
\newblock Dedicated to the memory of Tom Wolff.

\bibitem{Lee2003}
S.~Lee.
\newblock Endpoint estimates for the circular maximal function.
\newblock {\em Proc. Amer. Math. Soc.}, 131(5):1433--1442, 2003.

\bibitem{LS2013}
S.~Lee and A.~Seeger.
\newblock Lebesgue space estimates for a class of {F}ourier integral operators
  associated with wave propagation.
\newblock {\em Math. Nachr.}, 286(7):743--755, 2013.

\bibitem{Luiro2017}
H.~Luiro.
\newblock The variation of the maximal function of a radial function.
\newblock {\verb+arxiv.org/abs/1702.00669+} (2017).

\bibitem{LM2017}
H.~Luiro and J.~Madrid.
\newblock The variation of the fractional maximal function of a radial
  function.
\newblock To appear in \emph{Int. Math. Res. Not.},
  {\verb+arxiv.org/abs/1710.07233+} (2017).

\bibitem{MSS1992}
G.~Mockenhaupt, A.~Seeger, and C.~D. Sogge.
\newblock Wave front sets, local smoothing and {B}ourgain's circular maximal
  theorem.
\newblock {\em Ann. of Math. (2)}, 136(1):207--218, 1992.

\bibitem{PPSS2017}
C.~P\'erez, T.~Picon, O.~Saari, and M.~Sousa.
\newblock Regularity of maximal functions on {H}ardy--{S}obolev spaces.
\newblock {\verb+arxiv.org/abs/1711.01484+} (2017).

\bibitem{Ramos2017}
J.~P. Ramos.
\newblock Sharp total variation results for maximal functions.
\newblock {\verb+arxiv.org/abs/1703.00362+} (2017).

\bibitem{RdF1986}
J.~L. Rubio~de Francia.
\newblock Maximal functions and {F}ourier transforms.
\newblock {\em Duke Math. J.}, 53(2):395--404, 1986.

\bibitem{Schlag1997}
W.~Schlag.
\newblock A generalization of {B}ourgain's circular maximal theorem.
\newblock {\em J. Amer. Math. Soc.}, 10(1):103--122, 1997.

\bibitem{SS1997}
W.~Schlag and C.~D. Sogge.
\newblock Local smoothing estimates related to the circular maximal theorem.
\newblock {\em Math. Res. Lett.}, 4(1):1--15, 1997.

\bibitem{Sogge1991}
C.~D. Sogge.
\newblock Propagation of singularities and maximal functions in the plane.
\newblock {\em Invent. Math.}, 104(2):349--376, 1991.

\bibitem{Stein1976}
E.~M. Stein.
\newblock Maximal functions. {I}. {S}pherical means.
\newblock {\em Proc. Nat. Acad. Sci. U.S.A.}, 73(7):2174--2175, 1976.

\bibitem{bigStein}
E.~M. Stein.
\newblock {\em Harmonic analysis: real-variable methods, orthogonality, and
  oscillatory integrals}, volume~43 of {\em Princeton Mathematical Series}.
\newblock Princeton University Press, Princeton, NJ, 1993.
\newblock With the assistance of Timothy S. Murphy, Monographs in Harmonic
  Analysis, III.

\bibitem{Tanaka2002}
H.~Tanaka.
\newblock A remark on the derivative of the one-dimensional
  {H}ardy-{L}ittlewood maximal function.
\newblock {\em Bull. Austral. Math. Soc.}, 65(2):253--258, 2002.

\bibitem{Wolff2000}
T.~Wolff.
\newblock Local smoothing type estimates on {$L^p$} for large {$p$}.
\newblock {\em Geom. Funct. Anal.}, 10(5):1237--1288, 2000.

\end{thebibliography}

\bibliographystyle{abbrv}

\end{document}